\documentclass[11pt]{amsart}

\usepackage{graphicx}
\usepackage{amsmath}
\usepackage{amscd}
\usepackage{amsfonts}
\usepackage{amssymb}
\usepackage[dvipsnames]{xcolor}
\usepackage{fullpage}
\usepackage{mathtools}
\usepackage{graphicx}
\usepackage{caption}
\usepackage{subcaption}
\usepackage{verbatim}
\usepackage[pagebackref,hypertexnames=false, colorlinks, citecolor=red,linkcolor=blue, urlcolor=red]{hyperref}
\usepackage[margin = 1.45in]{geometry}

\numberwithin{equation}{section}

\theoremstyle{plain}
\newtheorem{theorem}{Theorem}[section]
\newtheorem{lemma}[theorem]{Lemma}
\newtheorem{proposition}[theorem]{Proposition}

\newtheorem{corollary}[theorem]{Corollary}

\theoremstyle{definition}
\newtheorem{definition}[theorem]{Definition}
\newtheorem{example}[theorem]{Example}
\newtheorem{remark}[theorem]{Remark}

\newtheorem{question}[theorem]{Question}

\newcommand{\be}{\begin{equation}}
\newcommand{\ee}{\end{equation}}
\newcommand{\bes}{\begin{equation*}}
\newcommand{\ees}{\end{equation*}}
\newcommand{\ba}{\begin{aligned}}
\newcommand{\ea}{\end{aligned}}

\newcommand{\cD}{\mathcal{D}}

\newcommand{\cO}{\mathcal{O}}
\newcommand{\cS}{\mathcal{S}}

\newcommand{\cW}{\mathcal{W}}

\newcommand{\bC}{\mathbb{C}}
\newcommand{\bD}{\mathbb{D}}
\newcommand{\bK}{\mathbb{K}}

\newcommand{\bR}{\mathbb{R}}

\newcommand{\Wmin}[1]{\cW^{\text{min}}_{#1}}
\newcommand{\Wmax}[1]{\cW^{\text{max}}_{#1}}

\newcommand{\saball}[1]{\mathcal{B}_{#1}}
\newcommand{\mixed}[2]{\mathcal{M}_{#1,#2}}
\newcommand{\FEP}[1]{\operatorname{CEP}(#1)}
\newcommand{\SB}[1]{\operatorname{SB}(#1)}

\begin{document}

\title{Complex Free Spectrahedra, Absolute Extreme Points, and Dilations}
\author{Benjamin Passer}

\abstract{Evert and Helton proved that real free spectrahedra are the matrix convex hulls of their absolute extreme points. However, this result does not extend to complex free spectrahedra, and we examine multiple ways in which the analogous result can fail. We also develop some local techniques to determine when matrix convex sets are not (duals of) free spectrahedra, as part of a continued study of minimal and maximal matrix convex sets and operator systems. These results apply to both the real and complex cases.
}}



\maketitle


\section{Introduction}

In the theory of classical convexity, the Krein-Milman theorem of \cite{KM1940} establishes that a compact convex set $K$ is the closed convex hull of its set of extreme points, which we denote $\text{ext}(K)$. However, if $K$ is a subset of $\bR^d$, then the closure is not necessary, as shown in Carath{\'e}odory's earlier theorem \cite{Car1907}. These facts may be expressed in terms of the set of affine functions on $K$, a function system which sits inside the continuous functions $C(K)$, and this perspective leads into the noncommutative point of view \cite{Arv69, Arv72, Choi_Effros}. 

A (concrete) operator system $\cS$ is a unital subspace of a $C^*$-algebra that is closed under the adjoint operation. The underlying structure of $\cS$ is contained in its positive elements $a \in \cS$ and, more importantly, the positive matrices $a = [a_{ij}] \in M_n(\cS)$ of any size. The morphisms between operator systems are unital completely positive (UCP) maps, which are linear maps preserving the unit and positivity of any matrix over $\cS$. Operator systems then provide the context for noncommutative versions of Choquet theory, through the study of pure UCP maps \cite{WW1999, Farenick00} and boundary representations \cite{DM05, Arv08, DK13}.

From the point of view of operator systems (whether finite-dimensional, separable, or nonseparable), the most appropriate notion of extreme point is a boundary representation \cite[Definition 2.1.1]{Arv69}. A boundary representation of $\cS$ is an irreducible representation $\pi: C^*(\cS) \to B(H)$ with the property that $\pi$ is the unique UCP extension of $\pi|_{\cS}$. By \cite[Proposition 2.4]{Arv08} and \cite[Proposition 2.2]{Arv_Notes}, $\pi$ has the unique extension property if and only if $\pi|_{\cS}$ is a maximal UCP map. That is, the only UCP dilations of the map are direct sums.
 Boundary representations are \lq\lq successful\rq\rq\hspace{0pt} extreme points since they are sufficient to completely norm any operator system by \cite[Theorem 3.4]{DK13}. More specifically, operator systems are completely normed by their pure UCP maps \cite{WW1999, Farenick00}, and every pure UCP map dilates to a boundary representation. However, this dilation may replace finite-dimensional maps with infinite-dimensional ones. 
 
 Many examples of operator systems arise in finite-dimensional contexts, in particular, when studying free spectrahedra or their polar duals \cite{EHKM, HKM13, EH}. In this case, one may be interested in a notion of extreme point that remains finite-dimensional. This is best illustrated in the setting of matrix convex sets. For any operator system $\cS$, one may consider the collection of UCP maps from $\cS$ into $M_n(\bC)$, where $n$ is arbitrary. For our discussion, $\cS$ will be finite dimensional, so we may consider a basis $I, T_1, T_1^*, \ldots, T_d, T_d^*$ (where perhaps $T_d^*$ is omitted if $T_d = T_d^*$). In this case, we will denote the operator system by $\cS_T$, and consider the matrix range \cite[\S 2.4]{Arv72}
\[ \cW(T) = \bigcup\limits_{n=1}^\infty \cW_n(T) = \bigcup\limits_{n=1}^\infty \{ (\phi(T_1), \ldots, \phi(T_d)): \phi: \cS_T \to M_n(\bC) \text{ is UCP}\}.\]
One may similarly start with a basis of self-adjoints, but this is merely a change of coordinate system. In any case, $\cW(T)$ is a prototypical closed and bounded matrix convex set over $\bC^d$.

\begin{definition}
Let $\mathcal{C} = \bigcup\limits_{n=1}^\infty \mathcal{C}_n$ be a subset of $\mathbb{M}(\bC)^d = \bigcup\limits_{n=1}^\infty M_n(\bC)^d$. A \textit{matrix convex combination} of points $X^{(i)} \in \mathcal{C}_{n_i}$ is any well-defined expression $Y$ of the form 
\[ Y = \sum\limits_{i=1}^k V_i^* X^{(i)} V_i, \hspace{.4 in}  V_i \in M_{n_i, m}(\bC), \hspace{.05 in} \sum\limits_{i=1}^k V_i^* V_i = I_{m}.\]
The set $\mathcal{C}$ is called \textit{matrix convex} if it is closed under matrix convex combinations. 
\end{definition}

Alternatively, $\mathcal{C}$ is matrix convex if and only if it is closed under direct sums and images UCP maps. That is,

\begin{itemize}
\item if $X \in \mathcal{C}_n$ and $Y \in \mathcal{C}_m$, then $X \oplus Y \in \mathcal{C}_{n+m}$, and 
\item if $X \in \mathcal{C}_n$ and $\phi: M_n(\bC) \to M_m(\bC)$ is UCP, then $\phi(X) \in \mathcal{C}_m$.
\end{itemize}
Matrix ranges $\cW(T)$ are always matrix convex, in addition to being closed and bounded in each level. Conversely, every matrix convex set over $\bC^d$ that is closed and bounded takes the form $\cW(T)$ for some tuple $T \in B(H)^d$ of bounded operators by \cite[Proposition 3.5]{DDSS}, as part of a larger duality between operator systems and matrix convex sets \cite{EW97}. 

In analogy with classical convexity, matrix convex sets can be studied through their extreme points. However, there are multiple notions of extreme point. A \textit{matrix extreme point} is a point $Y \in \mathcal{C}$ such that if $Y$ is written as a matrix convex combination $\sum\limits_{i=1}^k V_i^* X^{(i)} V_i$ such that each $V_i$ is surjective (which implies $\text{dim}(X^{(i)}) \leq \text{dim}(Y)$), then each $X^{(i)}$ is unitarily equivalent to $Y$. Matrix extreme points are sufficient to generate $\mathcal{C}$ as a closed and bounded matrix convex set by \cite[Theorem 4.3]{WW1999}, but there can be redundancy in the collection of matrix extreme points. A stronger notion is an \textit{absolute extreme point} (see \cite{Kleski14}), a point $Y$ such that any expression of $Y$ as a matrix convex combination $\sum\limits_{i=1}^k V_i^* X^{(i)} V_i$ with $V_i \not= 0$ has the property that $Y$ is unitarily equivalent to $X^{(i)}$ or a summand of $X^{(i)}$. There exist closed and bounded matrix convex sets with no absolute extreme points by \cite[Corollary 1.1]{Evert} or \cite[Example 6.30]{Kriel}. 

Matrix extreme and absolute extreme points may themselves be viewed in the language of operator systems. Indeed, if $\mathcal{C} = \cW(T)$, then a matrix extreme point is exactly an image of $T$ under a pure UCP map $\phi: \cS_T \to M_n(\bC)$ by \cite[Theorem B]{Farenick00}. Similarly, an absolute extreme point is exactly an image of $T$ under a (finite-dimensional) boundary representation $\pi: C^*(\cS_T) \to M_n(\bC)$ by \cite[Corollary 6.28]{Kriel}. Thus, if absolute extreme points are insufficient to generate $\cW(T)$, this simply means that $\cS_T$ is not completely normed by its \textit{finite-dimensional} boundary representations. Further, applying this dictionary to \cite[Lemma 2.3 and Theorem 2.4]{DK13} derives \cite[Lemma 6.12]{Kriel} -- a matrix extreme point that is not absolute extreme may be nontrivially dilated to another matrix extreme point. Similarly, \cite[Theorem 1.1 (3)]{EHKM} states that an irreducible tuple in $\mathcal{C}$ is an absolute extreme point precisely when it admits no nontrivial dilations in $\mathcal{C}$.

A free spectrahedron is a matrix convex set determined by a linear matrix inequality (see section \ref{sec:AEPcomplexFS}). Real free spectrahedra, that is, free spectrahedra whose coefficient matrices from a self-adjoint presentation have real coefficients, are considered in \cite{EH}. In \cite[Theorem 1.1]{EH}, it is shown that real free spectrahedra are the matrix convex hulls of their absolute extreme points. This is a success of the finite-dimensional point of view, in that a finite-dimensional problem has a finite-dimensional solution. Their result also applies to complex free spectrahedra that are closed under entrywise complex conjugation. In particular, \cite[Theorem 1.2]{EH} shows that a complex free spectrahedron that is closed under conjugation is associated to a real free spectrahedron in such a way that preserves the absolute extreme points. However, their result does not apply to general complex free spectrahedra, as in Remark \ref{remark:Cuntz}.

The purpose of this manuscript is to study absolute extreme points, boundary representations, and dilation theory without restricting to the real coefficient field. Section \ref{sec:AEPcomplexFS} shows that absolute extreme points of complex free spectrahedra are more delicate than their real counterparts, using both examples from the literature and new examples. Section \ref{sec:specdetect} then includes geometric characterizations that distinguish absolute extreme points of matrix convex sets in the first level. These results do not depend on the coefficient field $\mathbb{R}$ or $\mathbb{C}$, or equivalently on whether the set is closed under complex conjugation, unlike similar results such as \cite[Proposition 6.1]{EHKM}. We may leverage these conditions to prove that certain sets are not (duals of) free spectrahedra.


\section{Absolute Extreme Points of Complex Free Spectrahedra} \label{sec:AEPcomplexFS}

A \textit{free spectrahedron} is a matrix convex set determined by linear matrix inequalities, as follows. Given $A \in M_k(\bC)^d$, one may construct the \textit{hermitian monic linear pencil}
\[ L_A(Z) :=  I - \text{Re}\left(\sum\limits_{j=1}^d A_ j\otimes Z_j \right), \]
where $\text{Re}$ denotes the self-adjoint part of a matrix. Now, $L_A(Z)$ may be evaluated at any tuple $Z \in M_n(\bC)^d$ so long as $I$ is interpreted as the identity matrix of the appropriate size. The free spectrahedron $\cD_A = \bigcup\limits_{n=1}^\infty \cD_A (n)$ is then defined by
\begin{equation}\label{eq:complexfreeSP} \cD_A (n) = \{ Z \in M_n(\bC)^d: L_A(Z) \geq 0\}. \end{equation}

We note that for this manuscript, the notation $\cD_A$ (with no other specifications) allows for both complex coefficients and non self-adjoint matrices. It is, however, common to adjust both of these aspects of free spectrahedra. Using the identity 
\[ \text{Re}(A_j \otimes Z_j) = \text{Re}(A_j) \otimes \text{Re}(Z_j) - \text{Im}(A_j) \otimes \text{Im}(Z_j), \]
one may set a tuple $B = (\text{Re}(A_1), -\text{Im}(A_1), \ldots, \text{Re}(A_d), -\text{Im}(A_j))$ of self-adjoints and note that
\[ Z = (X_1 + iY_1, \ldots, X_d + iY_d) \in \cD_A \hspace{.2 in} \iff \hspace{.2 in} L_B(X_1, Y_1, \ldots, X_d, Y_d) \geq 0. \]
Thus, one may also set a \textit{self-adjoint} presentation of free spectrahedra. We use $g$ to denote the number of self-adjoint matrix variables and note that if $B \in M_k(\bC)^g_{sa}$, then the linear pencil takes a simplified form 
\[ L_B(X) = I - \sum\limits_{j=1}^g B_j \otimes X_j \]
for inputs $X \in M_n(\bC)^g_{sa}$ that are also self-adjoint. At this point, one may also restrict both $B$ and $X$ to self-adjoint matrix tuples that have coefficients in the real field.

\begin{definition}\label{def:realcomplexspectrahedron}
Let $\bK$ be either the real or complex field. Then for a tuple $B \in M_k(\bK)^g_{sa}$ of self-adjoint matrices with entries in $\bK$, the $\bK$-\textit{free spectrahedron} $\cD^\bK_B$ is defined by 
\[ \cD^\bK_B(n) = \{ X \in M_n(\bK)^g_{sa}: L_B(X) \geq 0 \}. \]
\end{definition}

\begin{remark}
We caution the reader that when we discuss the coefficient field of a free spectrahedron, or of the coefficient matrices, this refers \textit{exclusively} to the self-adjoint presentation. Similarly, it is the self-adjoint presentation that matters when one asks if a complex free spectrahedron is closed under complex conjugation in each entry of each matrix. Our main examples of problematic complex free spectrahedra (that are not real free spectrahedra) are such that a non self-adjoint form (\ref{eq:complexfreeSP}) comes from coefficient matrices $A_1, \ldots, A_d$ with real coefficients. However, the self-adjoint presentation has coefficient matrices with non-real entries.
\end{remark}

Any \textit{self-adjoint} matrix $X \in M_n(\bC)_{sa}$ decomposes as $X = R + iS$ where $R, S \in M_n(\bR)$ are real symmetric and real antisymmetric, respectively. The natural replacement for $R + iS$ is the real symmetric matrix $\begin{bmatrix} R & S \\ -S & R \end{bmatrix}$, but this fails to distinguish the roles of $R + iS$ and $R - iS$.

\begin{lemma}\cite[Lemma 3.3]{EH}
Let $\cD^\bC_B$, $B \in M_k(\bC)^g_{sa}$, be a complex free spectrahedron of $g$ self-adjoint variables. Then there exists a real free spectrahedron $\cD^\bR_C$, $C \in M_{2k}(\bR)^{2g}_{sa}$, of $2g$ real symmetric variables such that
\[ (R_1, S_1, \ldots, R_g, S_g) \in \cD^\bR_C \hspace{.2 in} \iff \hspace{.2 in} (R_1 + iS_1, \ldots, R_g + i S_g) \in \cD^\bC_B \]
precisely when $\cD^\bC_B$ is closed under entrywise complex conjugation.
\end{lemma}

The resolution of the spanning problem for \textit{real} free spectrahedra, or equivalently complex spectrahedra closed under complex conjugation, is given by \cite[Theorem 1.1]{EH}.

\begin{theorem} \cite[Theorem 1.1]{EH}
Any real free spectrahedron  $\cD^\bR_B$, or equivalently any complex free spectrahedron closed under entrywise complex conjugation, is the matrix convex hull of its absolute extreme points. Moreover, the absolute extreme points are a minimal spanning set in the sense that any closed collection of irreducible matrix tuples whose matrix convex hull is $\cD^\bR_B$, must include the absolute extreme points.
\end{theorem}

The above theorem is accompanied by explicit dilation algorithms and tight dimension bounds, and we note that the matrix convex hull does not require a closure. However, the case of complex free spectrahedra in full generality was left open. Our first remark is that counterexamples already exist in this case, without having been noted explicitly, using well-known $C^*$-algebraic constructions.

\begin{remark}\label{remark:Cuntz} As in \cite[Example 6.30]{Kriel}, the operator system generated by Cuntz isometries $S_1, ..., S_d$, $d \geq 2$, cannot have finite-dimensional boundary representations, as the Cuntz algebra $\cO_d$ has no finite-dimensional representations at all \cite{Cuntz}. Therefore, $\cW(S_1, \ldots, S_d)$ has no absolute extreme points, as absolute extreme points correspond to boundary representations by \cite[Theorem 4.2]{Kleski14}. 

On the other hand, it follows from \cite[Proposition 2.6]{Popescu} (see also \cite[Theorem 3.8]{Zheng}) that $\cW(S_1, S_2, \ldots, S_d)$ consists of matrix tuples $(T_1, ..., T_d)$, with the $T_i$ not necessarily self-adjoint, such that $[T_1 \,\,\, T_2 \,\,\, \ldots \,\,\, T_d]$ is a contraction. The set of row contractions is known to be a free spectrahedron; for example, it is an example of a spectraball in \cite[\S 1]{HKMV}.
\end{remark}

When written in self-adjoint coordinates $T_j = X_j + i Y_j$, the set of length $d$ row contractions has the form
\[ \left\{(X_1, Y_1, \ldots, X_d, Y_d) \in \mathbb{M}(\bC)^{2d}_{sa}:  \sum\limits_{j=1}^d (X_j + i Y_j)(X_j - i Y_j) \leq I \right\}. \]
We remind the reader that placing a free spectrahedron in its self-adjoint presentation is necessary in order to consider the coefficient field. There are many examples that show this free spectrahedron is not closed under complex conjugation, such as
\begin{equation}\label{eq:nonconjclosed} \left(\begin{bmatrix} 0 & 1/2 \\ 1/2 & 0 \end{bmatrix}, \begin{bmatrix} 0 & -i/2 \\ i/2 & 0 \end{bmatrix}, \begin{bmatrix} 0 & 0 \\ 0 & 1 \end{bmatrix}, \begin{bmatrix} 0 & 0 \\ 0 & 0 \end{bmatrix}, \ldots \right). \end{equation}
This example comes from writing the non self-adjoint pair $(T_1, T_2) = (E_{1,2}, E_{2,2})$ into self-adjoint coordinates. In this case, since each $T_j$ has real coefficients, its decomposed pieces have purely real/imaginary coefficients, such that complex conjugation corresponds to the mapping $T_j \mapsto T_j^*$. However, taking the coordinatewise adjoint of a row contraction does not always result in a row contraction.

The above example is distinct from the self-adjoint matrix ball,
\[ \saball{g} = \left\{ X \in \mathbb{M}(\bC)^g_{sa}:  \sum\limits_{k=1}^g X_k^2 \leq I\right\}, \]
which is closed under complex conjugation and hence fits into \cite[Theorem 1.1]{EH}. Our main object of study for this section is the free spectrahedron that results from combining the matrix ball with the set of row contractions. That is, we consider row contractions where some, but not all, of the coordinates are self-adjoint.

\begin{definition}
Given $d \geq 0$ and $g \geq 0$, not both zero, let
\[ \mixed{d}{g}(n) = \left\{ (T_1, \ldots, T_d, X_1, \ldots, X_g) \in M_n(\bC)^{d + g}:  X_k = X_k^*, \sum\limits_{j=1}^d T_j T_j^* + \sum\limits_{k=1}^g X_k^2 \leq I \right\}.\]
\end{definition}

Note that regardless of $d$ and $g$, $\mixed{d}{g}$ is a free spectrahedron for precisely the same reasons that apply to previous examples. Namely, writing $T_j = V_j + iW_j$ in self-adjoint coordinates gives that

\[ \sum\limits_{j=1}^d (V_j + i W_j)(V_j - i W_j) + \sum\limits_{k=1}^g X_k^2 \leq I  \iff || \begin{bmatrix} V_1 + i W_1 & \ldots & V_d + i W_d & X_1 & \ldots & X_g\end{bmatrix} || \leq 1 \]

\begin{equation*}\begin{aligned} 
\hspace{.8 in} &\iff -I \leq  \begin{bmatrix} 0 & V_1 + iW_1 & \ldots & V_d + i W_d & X_1 & \ldots& X_g \\ V_1 - iW_1 \\ \vdots \\ V_d - i W_d \\ X_1 \\ \vdots \\ X_g \end{bmatrix} \leq I \\
\hspace{.8 in} &\iff \sum\limits_{j=1}^d V_j \otimes A_j + \sum\limits_{j=1}^d W_j \otimes B_j + \sum\limits_{k=1}^g X_j \otimes C_k \leq I,
\end{aligned} \end{equation*}
where 

\begin{equation*} \begin{aligned}
A_j &= (E_{1, j + 1} + E_{j + 1, 1}) \oplus (-E_{1, j + 1} - E_{j + 1, 1}), \\
B_j &= (i \, E_{1, j + 1} -i \, E_{j + 1, 1}) \oplus (-i \, E_{1, j + 1} +i \, E_{j + 1, 1}), \\
C_k &= (E_{1, k + d + 1} + E_{k + d + 1, 1}) \oplus (-E_{1, k + d + 1} - E_{k + d + 1, 1}).
\end{aligned} \end{equation*}

Moreover, the set $\mixed{d}{g}$ is closed under complex conjugation in precisely two cases: if $d = 0$ and $g$ is arbitrary, or if $d = 1$ and $g = 0$. In all other cases, one may construct tuples similar to (\ref{eq:nonconjclosed}). As such, we will only consider when $d \geq 1$, in which case we are interested in finding the maximal elements of $\mixed{d}{g}$. That is, we seek elements $(T_1, \ldots, T_d, X_1, \ldots, X_g)$ of $\mixed{d}{g}$ with the property that the only dilations 
\[ \left( \begin{bmatrix} T_1 & * \\ * & * \end{bmatrix}, \ldots, \begin{bmatrix} T_d & * \\ * & * \end{bmatrix} , \begin{bmatrix} X_1 & * \\ * & * \end{bmatrix}, \ldots, \begin{bmatrix} X_g & * \\ * & * \end{bmatrix} \right) \]
that remain in $\mixed{d}{g}$ are the trivial dilations, namely, direct sums. Following \cite{DK19}, we refer to $d$-tuples of infinite-dimensional operators as members of \lq\lq level infinity\rq\rq\hspace{0pt} of the matrix convex set. We need the following lemma, which has a straightforward proof.

\begin{lemma}\label{lem:triv}
Fix $d \geq 1$ and $T_1, \ldots, T_d \in B(H)$. The following are equivalent.

\begin{enumerate}
    \item Each $T_i$ is injective, and the ranges of the $T_i$ are linearly independent subspaces.
    \item The only solution to $\sum\limits_{i=1}^d T_i W_i = 0$ is to have each $W_i = 0$.
\end{enumerate}
\end{lemma}

\begin{theorem}\label{thm:max_elements}
Fix $d \geq 1$ and $g \geq 0$, and consider $\mixed{d}{g}$. An operator tuple (possibly infinite-dimensional) $(T_i, X_j) = (T_1, \ldots, T_d, X_1, \ldots, X_g) \in \mixed{d}{g}$ is maximal if and only if all of the following conditions are met.

\begin{enumerate}
    \item\label{cond:ss} $\sum\limits_{i=1}^d T_iT_i^* + \sum\limits_{j=1}^g X_j^2 = I$.
    \item\label{cond:inj} Each $T_i$ is injective. 
    \item\label{cond:ran} The ranges of the $T_i$ are linearly independent subspaces.
\end{enumerate}
\end{theorem}

\begin{proof}
Suppose the three conditions hold, and that $\left( \begin{bmatrix} T_i & A_i \\ B_i & C_i\end{bmatrix}, \begin{bmatrix} X_j & Y_j \\ Y_j^* & Z_j \end{bmatrix}  \right)$ is a dilation of $(T_i, X_j)$ that remains in $\mixed{d}{g}$. Computing the sum-square of this tuple (which must be a contraction) and examining the top-left block shows that $\sum A_iA_i^* + \sum Y_jY_j^* = 0$, so each $A_i$ and $Y_j$ is zero. We now know that the sum-square is of the form
\[ \begin{bmatrix} I & \sum\limits_{i=1}^d T_i B_i^* \\ \sum\limits_{i=1}^d B_i T_i^* & * \end{bmatrix}, \]
and since this operator must still be a contraction, we have $\sum\limits_{i=1}^d T_i B_i^* = 0$. Lemma \ref{lem:triv} then implies that each $B_i$ is zero. We have shown the dilation is trivial, so $(T_i, X_j)$ is maximal.

To prove the conditions are all necessary, consider them separately. If condition (\ref{cond:ss}) fails, let $A = \sqrt{I - \sum T_i T_i^* - \sum X_j^2} \not= 0$, and note that the nontrivial dilation $\left( \begin{bmatrix} T_1 & A \\ 0 & 0\end{bmatrix}, \begin{bmatrix} T_2 & 0 \\ 0 & 0 \end{bmatrix}, \ldots, \begin{bmatrix} X_g & 0 \\ 0 & 0 \end{bmatrix} \right)$ remains in $\mixed{d}{g}$.

If either condition (\ref{cond:inj}) or (\ref{cond:ran}) fails, then using Lemma \ref{lem:triv}, we may find $B_1, \ldots, B_d \in B(H)$, at least one of which is nonzero, such that $\sum\limits_{i=1}^d T_i B_i^* = 0$. Using scalar multiplication if necessary, we may suppose $\sum\limits_{i=1}^d B_iB_i^* \leq I$. Direct computation then shows that the sum-square of $\left( \begin{bmatrix} T_i & 0 \\ B_i & 0 \end{bmatrix}, \begin{bmatrix} X_j & 0 \\ 0 & 0 \end{bmatrix}  \right)$ is still a contraction, so the tuple is a nontrivial dilation.
\end{proof}

\begin{remark}
If $d \geq 2$ and $g \geq 1$, the operators $T_i$ in a maximal tuple $(T_i, X_j) \in \mixed{d}{g}$ need not be isometric, and this can be shown by direct construction whenever $X_1 \not= 0$. 

However, if $d \geq 2$ and $g = 0$, the above theorem must produce only Cuntz isometries as the maximal elements, as this case is already known from \cite{Popescu} and \cite{Zheng}. The equivalence follows from some algebra: if $T_1, \ldots, T_d$ are injective operators with linearly independent ranges such that $T_1T_1^* + \ldots + T_dT_d^* = I$, then right multiplying by $T_1$ and rearranging shows that $T_1(T_1^*T_1 - I) + T_2(T_2^*T_1) + \ldots + T_d(T_d^* T_1) = 0$. By Lemma \ref{lem:triv}, $T_1^*T_1 - I = 0$ and $T_j^*T_1 = 0$ for $j \not= 1$. Applying the same trick to any $T_i$ in place of $T_1$ shows each $T_i$ is an isometry, and the ranges of these isometries are orthogonal.

Finally, if $d = 1$ and $g \geq 0$, the range condition (\ref{cond:ran}) is automatic. In particular, if $d = 1$ and $g = 0$, then $\mixed{1}{0}$ is the set of contractions, and the maximal contractions are known to be the unitaries. This is again recovered from Theorem \ref{thm:max_elements}: the first two conditions imply that $T_1$ is an injective (left invertible) operator with $T_1T_1^* = I$. Since that equation shows $T_1$ is surjective, we have that $T_1$ is actually invertible, and $T_1^*$ is its inverse. That is, $T_1$ is unitary. Note that in this case, the boundary representations of the corresponding operator system are $1$-dimensional.

\end{remark}

Since absolute extreme points correspond to finite-dimensional, irreducible, maximal elements, Theorem \ref{thm:max_elements} implies that there exists a complex free spectrahedron that admits a Krein-Milman theorem but not a Carath{\'e}odory theorem for absolute extreme points.

\begin{corollary}\label{cor:closure_needed_son}
If $d = 1$ and $g \geq 1$, then $\mixed{1}{g}$ is the closed matrix convex hull, but not the matrix convex hull, of its absolute extreme points. Consequently, the corresponding operator system is completely normed by its finite-dimensional boundary representations.
\end{corollary}
\begin{proof}
It suffices to prove the claims about absolute extreme points. Suppose an arbitrary matrix tuple $(T, X_1, \ldots, X_g)$ of $\mixed{1}{g}$ is given, so $TT^* + \sum\limits_{k=1}^g X_k^2 \leq I$. We may first approximate the tuple to arbitrary precision in order to assume that $TT^* + \sum\limits_{k=1}^g X_k^2 \leq (1 - \varepsilon) I$, after which we may approximate further to assume that $T$ is an invertible matrix. Next, write $TT^* + \sum\limits_{k=1}^g X_k^2 = I - AA^*$, and consider dilations of the form 
\[ (S, Y_1, \ldots, Y_g) = \left(\begin{bmatrix} T & A \\ B & C\end{bmatrix}, \begin{bmatrix} X_1 & 0 \\ 0 & D\end{bmatrix},  \begin{bmatrix} X_2 & 0 \\ 0 & 0 \end{bmatrix}, \ldots, \begin{bmatrix} X_g & 0 \\ 0 & 0 \end{bmatrix} \right).\]
Let $C$ be a positive invertible matrix, chosen small enough in norm that $B^* := -T^{-1} A C^*$ satisfies $BB^* + CC^* \leq I$. Since $T$ and $C$ are both invertible, it follows that 
\[ SS^* = \begin{bmatrix} TT^* + AA^* & TB^* + A C^* \\ BT^* + CA^* & BB^* + CC^* \end{bmatrix} = \begin{bmatrix} TT^* + AA^* & 0 \\ 0 & BB^* + CC^*\end{bmatrix}\] 
dominates a positive multiple of the identity and is therefore invertible. That is, the matrix $S$ is invertible. Setting $D := \sqrt{I - BB^* - CC^*}$, which is self-adjoint, gives that $SS^* + \sum\limits_{k=1}^g Y_k^2 = I$. It follows from Theorem \ref{thm:max_elements} that $(S, Y_1, \ldots, Y_k)$ is maximal, hence it is a direct sum of absolute extreme points. All together, we have that $\mixed{1}{g}$ is the closed matrix convex hull of its absolute extreme points.

On the other hand, consider $(0, \ldots, 0, 1) \in \mixed{1}{g}$, and suppose this point is a matrix convex combination of absolute extreme points $G^{(i)}$ of dimension $n_i$. Write
\begin{equation}\label{eq:siyivi} (0, \ldots, 0, 1) = \sum\limits_{i=1}^k V_i^* \, G^{(i)} \, V_i  \end{equation}
where the $V_i$ are $n_i \times 1$ matrices with $\sum\limits_{i=1}^k V_i^* V_i = 1 \in \mathbb{C}$. Now, each $V_i$ is just a contractive vector, so (\ref{eq:siyivi}) is a convex combination of vector states applied to $G^{(i)}$. Since the first level of $\mixed{1}{g}$ is a Euclidean ball, $(0, \ldots, 0, 1)$ is an extreme point of that set, so the above convex combination gives that $(0, \ldots, 0, 1)$ is the compression of a single absolute extreme point $G^{(i)}$ to a $1$-dimensional subspace. 

After a change of basis, we may write some absolute extreme point $(S, Y_1, \ldots, Y_g)$ of $\mixed{1}{g}$ as $S = \begin{bmatrix} 0 & \alpha \\ \beta & W \end{bmatrix}$, $Y_k = \begin{bmatrix} 0 & \gamma_k \\ \gamma_k^* & P_k \end{bmatrix}$ for $1 \leq k \leq d - 1$, and $Y_d = \begin{bmatrix} 1 & \gamma_d \\ \gamma_d^* & P_d \end{bmatrix}$. Since $SS^* + \sum\limits_{k=1}^g Y_k^2 = I$, the top left corner of the square sum implies that $\alpha = 0$ and $\gamma_k = 0$ for all $k$. However, this means that $S$ has a row of zeroes and is not invertible. Since $S$ acts on a finite-dimensional space, it is not injective. This contradicts Theorem \ref{thm:max_elements}, since every absolute extreme point is a maximal element.
\end{proof}

The results \cite[Theorem 6.8]{Kriel} and \cite[Theorem 1.10]{HL} imply that for any closed and bounded matrix convex set $\mathcal{C} = \cW(T)$ over $\bC^d$, the closure in the Webster-Winkler Krein-Milman theorem \cite[Theorem 4.3]{WW1999} is not necessary. That is, $\mathcal{C}$ is the matrix convex hull of its matrix extreme points, which are exactly the images of $T$ under pure matrix states. Since every pure UCP map of $\cS_T$ dilates to a boundary representation by \cite{DK13}, this implies that every point of $\mathcal{C}$ is a matrix convex combination of finite-dimensional compressions of expressions $\pi(T)$, where $\pi$ ranges over (perhaps infinite-dimensional) boundary representations. That is, the closure in Corollary \ref{cor:closure_needed_son} is needed only because the boundary representations being considered are finite-dimensional.

When $d \geq 2$, regardless of $g$, $\mixed{d}{g}$ follows a similar pattern as the row contractions. In particular, these are examples of complex free spectrahedra with no Krein-Milman theorem for absolute extreme points.

\begin{corollary} If $d \geq 2$ and $g \geq 0$, then $\mixed{d}{g}$ has no absolute extreme points. 
\end{corollary}
\begin{proof}
Any two injective operators $T_1$ and $T_2$ on a finite-dimensional space are automatically surjective, and hence their ranges cannot be linearly independent.
\end{proof}

However, the $C^*$-envelope is not simple, unlike the Cuntz algebra.

\begin{corollary}
If $d \geq 0$ and $g \geq 1$, then the $C^*$-envelope of the operator system corresponding to $\mixed{d}{g}$ is not simple.
\end{corollary}
\begin{proof}
The case $d = 0$, $g \geq 1$ is trivial. Otherwise, realize the operator system as concretely spanned by operators $\mathcal{T}_i, \mathcal{X}_j$. The conditions in Theorem 2.8 allow (possibly infinite-dimensional) maximal elements $(T_i, X_j)$ such that $\|X_j\|$ can take any value between $0$ and $1$. However, a maximal element is the image of $(\mathcal{T}_i, \mathcal{X}_j)$ under a representation, so not all of these representations are norm-preserving.
\end{proof}


\section{Geometry of Absolute Extreme Points} \label{sec:specdetect} 

If $\mathcal{C}$ is a closed and bounded matrix convex set, then certain extreme points of $\mathcal{C}_1 = K$ are absolute extreme points of $\mathcal{C}$. In an extreme case, \cite[Proposition 6.1]{EHKM} implies that every real free spectrahedron, equivalently every complex free spectrahedron that is closed under complex conjugation, has the property that \textit{all} extreme points of the first level are actually absolute extreme. (Again, we caution the reader that these results require the free spectrahedron to be given in self-adjoint coordinates before discussing complex conjugation of the entries.) However, for sets that are not spectrahedra, the situation is quite different. The first level of the matrix range of $\left( \begin{bmatrix} 1 & 0 \\ 0 & -1 \end{bmatrix}, \begin{bmatrix} 0 & 1 \\ 1 & 0 \end{bmatrix} \right)$ is the unit disk (see \cite[(4.11)]{P_SSM} and the references discussed nearby), but none of the extreme points of the disk are absolute extreme for this matrix range. Namely, the anticommuting pair above is a nontrivial dilation of every point of the unit circle.

In the general setting, we seek a geometric method to pick out certain absolute extreme points that are contained in the first level. From \cite[Proposition 4.3]{PScomp}, we have that isolated extreme points of $K = \mathcal{C}_1$ are absolute extreme points of $\mathcal{C}$. More generally, this result applies to any $\lambda \in K$ that is the vertex of some polytope $P$ that contains $K$. In this section, we significantly generalize \cite[Proposition 4.3]{PScomp} and use the new results to detect when certain sets are not polar duals of spectrahedra. Our point of view differs from that of \cite{EHKM} in two ways. First, we focus on the polar dual of spectrahedra instead of the spectrahedra themselves, and second, we do not need to restrict to real coefficients in any way. We denote the polar dual of $\mathcal{C}$ by $\mathcal{C}^\circ$, which has each level determined by
\[ \mathcal{C}^\circ_n = \left\{ (A_1, \ldots, A_g) \in M_n(\mathbb{C})^g_{sa}: \, \text{for all } X \in \mathcal{C}, \, \sum\limits_{j=1}^g A_j \otimes X_j \leq I \right\}. \]
See \cite{HKM17} for an extensive study of the polar dual in the real coefficient setting. Note that our choice to recoordinatize tuples to contain self-adjoints in proofs is only for convenience, and once a self-adjoint presentation is specified, it does not matter if the matrices have real or complex coefficients.

Given a compact convex set $K \subseteq \bR^g$, there are minimal and maximal matrix convex sets over $K$, denoted $\Wmin{}(K)$ and $\Wmax{}(K)$, respectively \cite[\S 4]{DDSS}. The disparity between these two objects can be used to measure to what extent the first level of a matrix convex set $\mathcal{C}$, given as $\mathcal{C}_1 = K$, determines the properties of $\mathcal{C}$. Note that while the operations of $\Wmin{}$ and $\Wmax{}$ in \cite{DDSS} are direct analogues of $\text{OMIN}$ and $\text{OMAX}$ for operator systems \cite{OMIN_OMAX}, a focus on matrix convex sets opens comparison problems up to a geometric point of view. The only compact convex sets $K \subseteq \bR^g$ with $\Wmin{}(K) = \Wmax{}(K)$ are simplices: see \cite[Theorem 4.7]{FNT}, \cite[Theorem 4.1]{PSS}, \cite[Theorem 1]{Huber}, and \cite[Corollary 2]{Aubrun}, which are roughly in increasing order of generality.

We will use the notation $\operatorname{AEP}(\mathcal{C})$ for the absolute extreme points of $\mathcal{C}$, $\text{ext}(K)$ for the Euclidean extreme points of $K$, and $\mathcal{I}_K$ for the isolated Euclidean extreme points of $K$. Note that a set of the form $\Wmin{}(K)$ has the property that any extreme point of $K$ is an absolute extreme point. Much more generally, \cite[\S 6]{Kriel} implies that if $\mathcal{C}$ is generated by its $n$th level, then any matrix extreme point in level $n$ is an absolute extreme point. Motivated by \cite[Proposition 4.3]{PScomp}, we give the following definitions.

\begin{definition}\label{def:SB}
Let $K$ be a compact convex set in Euclidean space. Then $\lambda \in K$ is a \textit{simplex-bounded point of} $K$ if there exists a simplex $\Delta \supseteq K$ such that $\lambda$ is a vertex of $\Delta$. Equivalently, there exists a polytope $P \supseteq K$ such that $\lambda$ is a vertex of $P$. In this case, we write $\lambda \in \SB{K}$. 
\end{definition}

\begin{definition}\label{def:FEP}
Let $\mathcal{C}$ be a closed and bounded matrix convex set, with $\mathcal{C}_1 = K$. Then $\lambda \in K$ is a \textit{commuting extreme point of} $\mathcal{C}$ if there exists a compact convex set $L$ such that $\mathcal{C} \subseteq \Wmin{}(L)$ and $\lambda$ is an extreme point of $L$. In this case, we write $\lambda \in \FEP{\mathcal{C}}$.
\end{definition}

In this language, \cite[Proposition 4.3]{PScomp} and its proof give that the containments
\begin{equation}\label{eq:foursets} \mathcal{I}_K \subseteq \SB{K} \subseteq \FEP{\mathcal{C}} \subseteq \operatorname{AEP}(\mathcal{C}) \cap K \end{equation}
hold. Note that a priori, $\SB{K}$ only depends on the first level, but $\FEP{\mathcal{C}}$ takes other levels into account. Among all $\mathcal{C}$ with $\mathcal{C}_1 = K$, the smallest collection of commuting extreme points occurs when $\mathcal{C} = \Wmax{}(K)$. Further, if one considers multiple commuting extreme points, the choice of $L$ may vary. Namely, if $\lambda_1$ and $\lambda_2$ are commuting extreme points of $\mathcal{C}$, it is not expected that there must be a single $L$ such that $\mathcal{C} \subseteq \Wmin{}(L)$ and both $\lambda_1$ and $\lambda_2$ are extreme points of $L$.

It is of interest to us under what circumstances the containments in (\ref{eq:foursets}) are or are not equalities. We first show that all of these sets coincide when $\mathcal{C} = \mathcal{W}(A)$ for a matrix tuple $A$. Note that $\mathcal{C} = \mathcal{W}(A)$ is the bounded polar dual of a free spectrahedron, as in the natural adjustment of \cite[Theorem 4.6]{HKM17} to the complex setting (see \cite[Proposition 3.1 and Lemma 3.2]{DDSS}).

\begin{proposition}\label{prop:equalityforall}
Let $A$ be a tuple of matrices and set $\mathcal{C} = \cW(A)$, $K = \mathcal{C}_1$. Then 
\[ \mathcal{I}_K = \SB{K} = \FEP{\mathcal{C}}= \operatorname{AEP}(\mathcal{C}) \cap K.\]
\end{proposition}
\begin{proof}
It suffices to prove that $\operatorname{AEP}(\mathcal{C}) \cap K \subseteq \mathcal{I}_K$. If $\lambda \in K$ is an absolute extreme point of $\cW(A)$, then by \cite[Corollary 3.8]{DP21}, $\lambda$ is a crucial matrix extreme point of $\cW(A)$ in the sense of \cite[Definition 2.4]{PScomp}. This implies that $\lambda$ is isolated in the extreme points of $K$.
\end{proof}

Proposition \ref{prop:equalityforall} is easily applied to the identification of free spectrahedra, through the polar dual. If $\mathcal{C}$ admits an absolute extreme point in level one that is not isolated among the Euclidean extreme points of level one, then $\mathcal{C}$ is not the matrix range of a matrix tuple. Consequently, it is not the polar dual of a free spectrahedron.

The containment (\ref{eq:foursets}) also provides a different interpretation of results such as \cite[Example 7.22]{DDSS}, which shows that the shifted unit disk $K = (1,0) + \overline{\mathbb{D}}$ is not scalable. That is,
\[ \forall \, C \in (0, \infty), \hspace{.1 in} \Wmax{}(K) \not\subseteq C \cdot \Wmin{}(K). \]
This example was subsequently generalized in \cite[Theorem 5.6 and Corollary 5.7]{PSS} to a characterization of all containments $\Wmax{}(K_1) \subseteq \Wmin{}(K_2)$ where $K_1$ and $K_2$ are shifted and scaled closed Euclidean balls -- there always exists a simplex $\Delta$ with $K_1 \subseteq \Delta \subseteq K_2$. However, nonscalability of the shifted disk also follows from the next result, which uses a well-known fact about numerical ranges.

\begin{corollary}\label{cor:noFEPball}
No point of the Euclidean ball $\overline{\mathbb{B}^g_2}$ in dimension $g \geq 2$ is a commuting extreme point of $\Wmax{}(\overline{\mathbb{B}^g_2})$.
\end{corollary}
\begin{proof}
Let $F = (F_1, \ldots, F_g)$ be the $g$-tuple of universal self-adjoint, anticommuting, unitary matrices. Then $\cW(F) \subseteq \Wmax{}(\overline{\mathbb{B}^g_2})$. Since any extreme point of the ball is a compression of $F$ to a non-reducing subspace, no $\lambda \in \overline{\mathbb{B}^g_2}$ is an absolute extreme point of $\Wmax{}(\overline{\mathbb{B}^g_2})$. By (\ref{eq:foursets}), $\operatorname{AEP}(\Wmax{}(\overline{\mathbb{B}^g_2})) \cap \overline{\mathbb{B}^g_2} = \varnothing$ implies that no $\lambda \in \overline{\mathbb{B}^g_2}$ is a commuting extreme point of $\Wmax{}(\overline{\mathbb{B}^g_2})$.
\end{proof}

This simple result recovers nonscalability of the tangential ball $K = \vec{v} + \overline{\mathbb{B}^g_2}$ where $g \geq 2$ and $||\vec{v}||_2 = 1$ immediately: if $\Wmax{}(K) \subseteq C \cdot \Wmin{}(K)$, then the extreme point $0$ of $K$ is still an extreme point of $C \cdot K$, and hence $0$ is a commuting extreme point of $K$ by Definition \ref{def:FEP}. However, shifting the set shows $-\vec{v}$ is a commuting extreme point of $\overline{\mathbb{B}^g_2}$, which is a contradiction. Note that Corollary \ref{cor:noFEPball} is a fundamentally distinct generalization of nonscalability than \cite[Theorem 5.6 and Corollary 5.7]{PSS}, in that neither result obviously implies the other. A quick look at the proof also shows that the argument can be replaced with a local version, as follows.

\begin{corollary}\label{cor:noFEPgen}
Let $K$ be a compact convex set, and suppose $\lambda \in K$ is such that there exists a closed Euclidean ball $B$ with $B \subseteq K$ and $\lambda \in B$. Then $\lambda$ is not an absolute extreme point of $\Wmax{}(K)$, hence it is not a commuting extreme point of $K$.
\end{corollary}

Corollary \ref{cor:noFEPgen} applies in particular to $\ell^p$-balls for $p \geq 2$, but not for $p < 2$. So, we use $\ell^p$ behavior for $p < 2$ to shed light on the containments in (\ref{eq:foursets}). To this end, we define the compact convex set $K_p \subseteq \mathbb{R}^2$ as
\begin{equation} K_p := \{(x, y) \in \mathbb{R}^2: |x|^p \leq y \leq 1 \}. \end{equation}
If $1 < p < 2$, then the extreme point $(0, 0)$ of $K_p$ is not simplex-bounded, but Corollary \ref{cor:noFEPgen} does not apply, so it is not immediately obvious if it is an absolute extreme point. 

\begin{example}\label{ex:nonSBAEP}
Fix $1 < p < 2$. Then $(0, 0) \in K_p$ is an absolute extreme point of $\Wmax{}(K_p)$ that is not a simplex-bounded point of $K_p$.
\end{example}
\begin{proof}
It is immediate that $(0, 0)$ is not simplex-bounded since $F(x) := |x|^p$ has $F^\prime(0) = 0$, so we need only show that $(0, 0)$ is an absolute extreme point of $\Wmax{}(K_p)$. Consider an arbitrary $2\times2$ dilation
\[ (X, Y) = \left( \begin{bmatrix} 0 & a \\ \overline{a} & b \end{bmatrix}, \begin{bmatrix} 0 & c \\ \overline{c} & d \end{bmatrix} \right) \in \Wmax{}(K_p). \] 
By definition, the numerical range of $(X, Y)$ is in $K_p$, so $Y \geq 0$. This immediately gives that $c = 0$ and $d \geq 0$, so
\[ (X, Y) = \left( \begin{bmatrix} 0 & a \\ \overline{a} & b \end{bmatrix}, \begin{bmatrix} 0 & 0 \\ 0 & d \end{bmatrix} \right).\] 
Further, for any unit vector $v = (v_1, v_2) \in \bC^2$, it holds that
\bes
\left| \langle Xv, v \rangle \right|^{p} \leq \langle Yv, v \rangle \leq 1,
\ees
and in particular
\bes
| 2  \operatorname{Re}(a v_2 \overline{v_1}) + b \, |v_2|^2|^{p} \leq d \, |v_2|^2.
\ees
Choosing the arguments of $v_1$ and $v_2$ carefully, we have that for any $t \in (0,1)$,
\bes
\left| 2 |a| t \sqrt{1-t^2} + b \, t^2 \right|^{p} \leq d t^2.
\ees
The equivalent expression
\bes
\left| 2 |a| t^{1-2/p} \sqrt{1-t^2} + b \, t^{2-2/p} \right|^{p} \leq d
\ees
holds, so the left hand side is bounded as $t$ decreases a zero. This causes a contradiction if $a \not= 0$, since $1 - 2/p < 0$ and $2 - 2/p > 0$. We conclude that $a = 0$, and hence the arbitrary dilation $(X, Y)$ was trivial. It follows from \cite[Theorem 1.1]{EHKM} that $(0, 0)$ is an absolute extreme point of $\Wmax{}(K)$.
\end{proof}

To generalize the example, we will need some preparation. Given a convex body $K$ in real or complex Euclidean space and an extreme point $\lambda$ of $K$, recoordinatize $K$ so that $\lambda = \vec{0} \in K$ and $K \subseteq \bR^{g-1} \times [0, \infty)$. Since $K$ is a convex body, we also insist that the set
\begin{equation} D := \{ \vec{x} \in \bR^{g-1}: \exists y \geq 0, (\vec{x}, y) \in K\} \end{equation}
has $\vec{0} \in \mathbb{R}^{g-1}$ in the interior. If all of these conditions are satisfied, we say that $K$ has been put in a (non-unique) \textit{standard position} around $\lambda$.

If $K$ is in a standard position around the extreme point $\vec{0}$, define the function $F: D \to [0, +\infty)$ by
\begin{equation}\label{eq:payrespects} F(\vec{x}) = \min\{y \geq 0:  (\vec{x}, y) \in K\}, \end{equation}
that is,
\begin{equation}\label{eq:simpler} (\vec{x}, y) \in K \,\,\,\, \implies \,\,\,\, F(\vec{x}) \leq y. \end{equation}
We will use the decay rate of $F(\vec{x})$ as $\vec{x}$ approaches $\vec{0}$ to generalize Example \ref{ex:nonSBAEP}. Note that placing $\mathcal{C}_1 = K$ in standard form means we have recoordinatized $K$ into real coordinates, and hence broken up the matrices of $\mathcal{C}$ into self-adjoint coordinates, but this difference is only bookkeeping.

\begin{theorem}\label{thm:subquadratic}
Let $K$ be a convex body in Euclidean space, with $\lambda \in K$ an extreme point, and put $K$ in a standard position around $\lambda$. If $\lim\limits_{\vec{x} \to 0} \, \cfrac{F(\vec{x})}{||\vec{x}||^2} = +\infty$, then $\lambda$ is an absolute extreme point of any closed and bounded matrix convex set $\mathcal{C}$ with $\mathcal{C}_1 = K$.
\end{theorem}
\begin{proof}
We prove the contrapositive. Suppose that $\lambda = \vec{0}$ is not an absolute extreme point, so there is a nontrivial $2\times2$ dilation
\[ (X_1, \ldots, X_{g-1}, Y) = \left( \begin{bmatrix} 0 & a_1 \\ \overline{a_1} & b_1 \end{bmatrix}, \ldots, \begin{bmatrix} 0 & a_{g-1} \\ \overline{a_{g-1}} & b_{d-1} \end{bmatrix}, \begin{bmatrix} 0 & \alpha \\ \overline{\alpha} & \beta \end{bmatrix} \right) \in \mathcal{C}_2.\] 
We note that using standard form has recoordinatized $\mathcal{C}$ so that each $X_i$ is self-adjoint and $Y \geq 0$, which immediately gives that $\alpha = 0$ and $\beta \geq 0$. The dilation is by assumption nontrivial, so at least one $a_i$ is nonzero, and WLOG we assume $a_1 \not= 0$.
If $v = (v_1, v_2)$ is an arbitrary unit vector, $v_2 \not= 0$, then the image of $(X_1, \ldots, X_{g-1}, Y)$ under this vector state belongs to $K$, and hence (\ref{eq:simpler}) gives that
\begin{equation}\label{eq:vecstate}
F(\langle X_1 v, v \rangle, \ldots, \langle X_{g-1} v, v \rangle) \leq \langle Y v, v \rangle. 
\end{equation}
Set $v_2 = t \in (0, 1]$, and choose $v_1$ such that 
\[ \langle X_1 v, v \rangle = 2 \operatorname{Re}(a_1 v_2 \overline{v_1}) + b_1 |v_2|^2 = 2 |a_1| t \sqrt{1 - t^2} + b_1 t^2. \]
If $\vec{x}(t) := (\langle X_1 v, v \rangle, \ldots, \langle X_{g-1} v, v \rangle) \in D$, then (\ref{eq:vecstate}) implies that $\cfrac{F(\vec{x}(t))}{t^2} \leq \beta$. 

Since $a_1 \not= 0$, we have that
\bes
\limsup_{t \to 0^+} \frac{t^2}{||\vec{x}(t)||^2} \leq \limsup_{t \to 0^+} \frac{t^2}{|x_1(t)|^2} = \limsup_{t \to 0^+} \frac{t^2}{\left| 2 |a_1| t \sqrt{1 - t^2} + b_1 t^2 \right|^2} = \frac{1}{4|a_1|^2},
\ees
so we conclude
\be
\limsup_{t \to 0^+} \frac{F(\vec{x}(t))}{||\vec{x}(t)||^2} \leq \frac{\beta}{4|a_1|^2}.
\ee
Finally, there is a path of points $\vec{x} = \vec{x}(t)$ approaching $0$ such that $\cfrac{F(\vec{x})}{||\vec{x}||^2}$ remains bounded, and $\liminf\limits_{\vec{x} \to \vec{0}} \cfrac{F(\vec{x})}{||\vec{x}||^2} < +\infty$.
\end{proof}

Note that the hypothesis of Theorem \ref{thm:subquadratic} implies that $x_g = 0$ is a separating hyperplane for $\lambda = \vec{0}$, so $\lambda$ is actually an exposed point. Theorem \ref{thm:subquadratic} improves \cite[Proposition 4.3]{PScomp} in that we may determine some $\lambda \in K$ is an absolute extreme point of any matrix convex set over $K$ using a weaker geometric assumption. However, we do not reach any information on whether $\lambda$ is a commuting extreme point, which was implicit in \cite[Proposition 4.3]{PScomp}. Combining Theorem \ref{thm:subquadratic} and Proposition \ref{prop:equalityforall} also improves the previous discussion about identifying free spectrahedra. If an exposed point $\lambda$ of $K$ is not isolated extreme, but the defining function of $K$ near $\lambda$ decays more slowly than the norm squared, then $K$ is not the polar dual of a spectrahedron. Our next result concerns the opposite type of behavior: if $\mathcal{C} = \cW(A)$, where $A$ is a tuple of matrices, then no defining function $F(\vec{x})$ near an exposed point decays strictly \textit{faster} than the norm squared.

\begin{theorem}\label{thm:superquadratic}
Let $K$ be a convex body in Euclidean space, and let $\lambda$ be an exposed point of $K$ that is separated by a hyperplane $Q$. If $K = \cW_1(A)$ for some matrix tuple $A$, then $K$ is contained in a paraboloid with vertex $\lambda$ that opens away from $Q$. That is, when $K$ is put in standard position around $\lambda$ with separating hyperplane $x_g = 0$, there is some $M > 0$ with $F(\vec{x}) \geq M \|x\|^2$ for all $\vec{x} \in D$.
\end{theorem}
\begin{proof}
Put $K$ in standard position with respect to $\lambda$, so that the separating hyperplane $Q$ corresponds to $x_g = 0$. This implies that the function $F$ of (\ref{eq:payrespects}) is such that for $\vec{x} \not= \vec{0}$, $F(\vec{x}) > 0$. 

Suppose $\cW_1(A) = K$ for a matrix tuple $A$. Then the final matrix $A_g$ is positive semidefinite, since $K$ is in standard form. Note that $A_g$ is not the zero matrix, as $A_g = 0$ would imply that every point of $K$ has $x_g = 0$. Similarly, $A_g$ is not positive definite, since $\vec{0} \in \cW_1(A)$. Apply a simultaneous unitary conjugation to $A$ so that $A_g = \begin{bmatrix} 0 & 0 \\ 0 & D \end{bmatrix}$, where $D \geq \varepsilon I$ is positive definite.

Since $\lambda = \vec{0}$ is exposed in $K$ and separated by the hyperplane $Q: x_g = 0$, any state that maps $A_g \mapsto 0$ must map $A_i \mapsto 0$ for all $i$. Therefore, since the unitary conjugation  above has produced a zero block corner for $A_g$, the other $A_i$ must have zero blocks in the same place. That is,
\bes
A_i = \begin{bmatrix} 0 & B_i \\ B_i^* & C_i \end{bmatrix} ,\hspace{.1 in} 1 \leq i \leq g - 1, \hspace{.75 in} A_g = \begin{bmatrix} 0 & 0 \\ 0 & D \end{bmatrix}.
\ees

Applying an arbitrary vector state corresponding to $(v_1, v_2)$, $\|v_1\|^2 + \|v_2\|^2 = 1$, we obtain a point $x = (x_1, \ldots, x_g) \in K$ where
\be\label{eq:vecpoint}
x_i = 2 \, \text{Re}\left( \langle B_i v_2, v_1 \rangle \right) + \langle C_i \, v_2, v_2 \rangle, \hspace{.1 in} 1 \leq i \leq g - 1, \hspace{.3 in} x_g = \langle D \, v_2, v_2 \rangle.
\ee
Since $x_g = \langle D \, v_2, v_2 \rangle$ and $D \geq \varepsilon I$, we have that $x_g \geq \varepsilon \|v_2\|^2$, equivalently $||v_2|| \leq \cfrac{\sqrt{x_g}}{\sqrt{\varepsilon}}$. It is also trivial that $||v_1|| \leq 1$ and $||v_2|| \leq 1$, so we may bound the other coordinates $x_i$, $1 \leq i \leq g - 1$, by the estimate
\bes \ba
|x_i| 	&= \left| 2 \, \text{Re}\left(\langle B_i v_2, v_1 \rangle \right) + \langle C_i \, v_2, v_2 \rangle \right| \\
      	&\leq 2 \, ||B_i|| \, ||v_2|| \, ||v_1|| + ||C_i|| \, ||v_2||^2 \\
	&\leq (2 ||B_i|| + ||C_i||) \, ||v_2|| \\
	&\leq \frac{2 ||B_i|| + ||C_i||}{\sqrt{\varepsilon}} \cdot \sqrt{x_g}.
\ea \ees
Letting $M^{-1} = \sum\limits_{i=1}^{g-1} \cfrac{ (2 ||B_i|| + ||C_i||)^2}{\varepsilon}$ shows that any image $(x_1, \ldots, x_g)$ of $A$ from a vector state satisfies $\sum\limits_{i=1}^{g-1} x_i^2 \leq M^{-1} x_g$, or rather $x_g \geq M \sum\limits_{i=1}^{g-1} x_i^2$. Since the inequality given determines a paraboloid, which is convex, it follows that $\mathcal{W}_1(A) = K$ is contained in the same set. In other words, we have $F(\vec{x}) \geq M ||\vec{x}||^2$.
\end{proof}

Since the unit disk is $\cW_1(A)$ for a pair $A$ of $2 \times 2$ matrices, the convex hull of two disks is $\cW_1(A \oplus B)$ for two pairs $A$ and $B$ of $2 \times 2$ matrices. This set has an extreme point that is not exposed. On the other hand, the vertex of any paraboloid is certainly an exposed point, so Theorem \ref{thm:superquadratic} cannot have its hypothesis weakened to apply to general extreme points.

\begin{example}
Let $\overline{\mathbb{B}^g_p}$ denote the closed real $\ell_p$ ball in dimension $g \geq 2$. By \cite[Theorems 2.2 and 3.1]{HV}, if $p \not\in \{1, 2, \infty\}$, then $\overline{\mathbb{B}^g_p}$ cannot be a spectrahedron (see also \cite[Example 2]{HV}). For most values of $p$, $\overline{\mathbb{B}^g_p}$ is not an algebraic interior: there is no polynomial $P(x)$ such that a connected component of $\{x: P(x) > 0\}$ has closure equal to $\overline{\mathbb{B}^g_p}$. In all other cases, the region fails the rigid convexity condition defined in \cite[\S 3.1]{HV}, as a generic line through the origin will not cross the unit ball the necessary number of times for $\overline{\mathbb{B}^g_p}$ to be defined by a linear matrix inequality. We note that the failure of $\overline{\mathbb{B}^g_p}$ to be a spectrahedron may also be seen through local information in the polar dual, as outlined below. 

If $1 < p < \infty$ with $p \not= 2$, and we assume $\overline{\mathbb{B}^g_p} = D_A(1)$ for some $A \in M_n(\bC)^g$, then the polar dual and \cite[Proposition 3.1 and Lemma 3.2]{DDSS} show $\overline{\mathbb{B}^g_q} = \cW_1(A)$, where $1/p + 1/q = 1$. If $1 < q < 2$, then the point $(1, 0, \ldots, 0)$ of $\overline{\mathbb{B}^g_q}$ is an absolute extreme point of $\mathcal{W}(A)$ by Theorem \ref{thm:subquadratic}, but it is not an isolated extreme point of $\overline{\mathbb{B}^g_q}$, which contradicts Proposition \ref{prop:equalityforall}. On the other hand, if $2 < q < \infty$, then $(1, 0, \ldots, 0)$ is exposed in $\overline{\mathbb{B}_q^g}$, but there is no paraboloid containing $\overline{\mathbb{B}_q^g}$ for which $(1, 0, \ldots, 0)$ is the vertex, which contradicts Theorem \ref{thm:superquadratic}.
\end{example}

We close with some further discussion of (\ref{eq:foursets}). We know that for matrix ranges of matrix tuples, the four sets coincide, and in the general case we have found an absolute extreme point in level one that is not simplex-bounded. That is, in general $\text{AEP}(\mathcal{C}) \cap K$ and $\SB{K}$ need not be equal. We still do not know whether either of these sets is equal to the set of commuting extreme points.

\begin{question}
Let $\mathcal{C} = \Wmax{}(K)$. Does there exist a choice of $K$ such that both of the equalities $\SB{K} = \FEP{\mathcal{C}}$ and $\operatorname{AEP}(\mathcal{C}) \cap K = \FEP{\mathcal{C}}$ fail?
\end{question}

The potential equality of the simplex-bounded points of $K$ and the commuting extreme points of $\Wmax{}(K)$ would be an interesting \lq\lq local\rq\rq\hspace{0pt} version of the claim that $\Wmax{}(K) = \Wmin{}(K)$ if and only if $K$ is a simplex. The set
\[ K_p = \{(x, y) \in \mathbb{R}^2: |x|^p \leq y \leq 1 \} \]
remains of interest for $1 < p < 2$, as it has an absolute extreme point $(0, 0)$ that is not simplex-bounded, but that point has not been classified as a commuting extreme point or not. Following the discussion after Corollary \ref{cor:noFEPball}, we consider a simpler question: is $K_p$ scalable? 

\begin{question}
Fix $p \in (1, 2)$ and let $K_p = \{(x, y) \in \bR^2: |x|^p \leq y \leq 1\}$. Does there exist a constant $M$ such that $\Wmax{}(K_p) \subseteq M \cdot \Wmin{}(K_p)$?
\end{question}

We have a partial answer to this question: for $p > 4/3$, $K_p$ is not scalable, and we prove this using disks inside $K_p$ that approach the origin.

\begin{lemma}\label{lem:p_ball_containment}
Fix $p \in (1, 2)$ and the set $K_p = \{(x, y): |x|^p \leq y \leq 1\}$. For sufficiently small $c > 0$, the closed disk of radius $c - (pc)^{\frac{p}{2 - p}}$ centered at $(0, c)$ is contained in $K_p$.
\end{lemma}
\begin{proof}
$K_p$ is convex, so we need only prove containment of the boundary circle. The disk of radius $r > 0$ centered at $(0, c)$ is contained in $K_p$ precisely when
\[ x^2 + (y - c)^2 = r^2 \,\,\,\,\,\,\,\,\, \implies \,\,\,\,\,\,\,\,\, |x|^p \leq y \leq 1. \]
The claim $|x|^p \leq y$ is the same as $x^2 \leq y^{2/p}$. If $x^2 + (y - c)^2 = r^2$, this means $r^2 - (y - c)^2 \leq y^{2/p}$. This gives us the equivalent implication
\[ x^2 + (y - c)^2 = r^2 \,\,\,\,\,\,\,\,\, \implies \,\,\,\,\,\,\,\,\, 0 \leq y \leq 1 \text{ and } y^{2/p} + (y - c)^2 \geq r^2. \]
The right hand side is independent of $x$, so we may focus on $y$ and $r$. We need $r > 0$ such that $r \leq c$ and $c + r \leq 1$, such that
\begin{equation}\label{eq:ball_containment}
c - r \leq y \leq c + r \,\,\,\,\,\,\,\,\, \implies \,\,\,\,\,\,\,\,\, y^{2/p} + (y - c)^2 \geq r^2.
\end{equation}

The function $f(y) := y^{2/p} + (y - c)^2$ has $f^\prime(y) = \frac{2}{p} y^{2/p - 1} + 2(y - c)$. Since $1 < p < 2$ implies $2/p - 1 > 0$, $f^\prime$ is increasing on $[0, +\infty)$. To find a region where $f^\prime(y) \geq 0$, we forgo solving $f^\prime(y) = 0$ and instead ignore the linear term. Solving $\frac{2}{p}y^{2/p - 1} - 2c = 0$ gives $y = (pc)^{\frac{p}{2 - p}}$, so we now have
\[ y \geq (pc)^{\frac{p}{2 - p}} \,\,\,\,\,\,\, \implies \,\,\,\,\,\,\,  f^\prime(y) \geq f^\prime( (pc)^{\frac{p}{2 - p}}) = 2(pc)^{\frac{p}{2 - p}} > 0, \]
and hence $f$ is increasing for $y \geq (pc)^{\frac{p}{2 - p}}$. We conclude that if $\delta := (pc)^{\frac{p}{2 - p}}$, then
\begin{equation}\label{eq:delta_ball} y \geq \delta \,\,\,\,\,\, \implies \,\,\,\,\,\,\, f(y) \geq f(\delta) \geq (\delta - c)^2 = (c - \delta)^2. \end{equation}
Since $\frac{p}{2 - p} > 1$, it holds that for sufficiently small $c$, $0 < \delta < c$ and $2c < 1$. If $r := c - \delta = c - (pc)^{\frac{p}{2 - p}}$, then we have that $0 < r < c$ and $c + r < 1$, and finally (\ref{eq:ball_containment}) holds as a consequence of (\ref{eq:delta_ball}). That is, the disk of radius $r = c - (pc)^{\frac{p}{2-p}}$ centered at $(0, c)$ is contained in $K_p$.
\end{proof}

\begin{theorem}\label{thm:Kpscale}
Fix $p \in (\frac{4}{3}, \infty)$. Then $K_p = \{(x, y) \in \mathbb{R}^2: |x|^p \leq y \leq 1\}$ is not scalable. That is, there is no $M > 0$ such that $\Wmax{}(K_p) \subseteq M \cdot \Wmin{}(K_p)$.
\end{theorem}
\begin{proof}
If $p \geq 2$, then $(0,0)$ is not a commuting extreme point of $K_p$ by Corollary \ref{cor:noFEPball}, so $K_p$ is not scalable. We may assume $4/3 < p < 2$.

With $p$ fixed, Lemma \ref{lem:p_ball_containment} shows that for sufficiently small $c > 0$, if $r := c - (pc)^{\frac{p}{2 - p}}$, then $(0, c) + r \overline{\bD} \subseteq K_p$. We also have that since $p < 2$, $K_p \subseteq K_2 \subseteq (0, 1) + \overline{\bD}$. If $K_p$ is scalable, then we may fix $M > 0$ such that $\Wmax{}(K_p) \subseteq M \cdot \Wmin{}(K_p)$, which implies that
\[ \Wmax{}( (0, c) + r \overline{\bD}) \subseteq \Wmin{}((0, M) + M \overline{\bD}). \]
The disk is a Euclidean ball of real dimension $d = 2$, so applying \cite[Theorem 5.6]{PSS} shows that
\[ M \geq \sqrt{(M - c)^2 + r^2} + r, \]
which simplifies to
\[ M \geq \frac{c^2}{2(c - r)} = \frac{c^2}{2(pc)^{\frac{p}{2 - p}}}. \]
Now, $4/3 < p < 2$ is fixed, so $\frac{p}{2 - p} > 2$, which implies that the right hand side grows arbitrarily large as $c$ decreases to $0$. This is a contradiction, since the inequality must hold for sufficiently small $c$.
\end{proof}

\section*{Acknowledgments}

This work was supported by the United States Naval Academy Jr. NARC program.

\bibliographystyle{amsplain}

\end{document}